\documentclass[12pt,reqno]{article}

\usepackage[usenames]{color}
\usepackage{amssymb}
\usepackage{graphicx}
\usepackage{amscd}

\usepackage[colorlinks=true,
linkcolor=webgreen,
filecolor=webbrown,
citecolor=webgreen]{hyperref}

\definecolor{webgreen}{rgb}{0,.5,0}
\definecolor{webbrown}{rgb}{.6,0,0}
\definecolor{red}{rgb}{1,0,0}

\def\Enn{\mathbb{N}}

\def\Preals{\mathbb{R}^{+}}
\def\pal{{\tt PAL}}
\def\apal{{\tt APAL}}
\def\suchthat{\, : \,}

\newcommand{\ceil}[1]{\left\lceil#1\right\rceil}
\newcommand{\abs}[1]{\left|#1\right|}

\def\qpal{Q_{\rm pal}}
\def\qapal{Q_{\rm apal}}

\usepackage{color}
\usepackage{fullpage}
\usepackage{float}

\usepackage{graphics,amsmath,amssymb}
\usepackage{amsthm}
\usepackage{amsfonts}
\usepackage{latexsym}
\usepackage{epsf}

\newcommand{\seqnum}[1]{\href{http://oeis.org/#1}{\underline{#1}}}

\DeclareMathOperator{\pald}{pald}

\begin{document}

\theoremstyle{plain}
\newtheorem{theorem}{Theorem}
\newtheorem{corollary}[theorem]{Corollary}
\newtheorem{lemma}[theorem]{Lemma}
\newtheorem{proposition}[theorem]{Proposition}
\newtheorem{hprinciple}{Heuristic Principle}

\theoremstyle{definition}
\newtheorem{definition}[theorem]{Definition}
\newtheorem{example}[theorem]{Example}
\newtheorem{conjecture}{Conjecture}
\newtheorem{openproblem}{Open Problem}

\theoremstyle{remark}
\newtheorem{remark}[theorem]{Remark}

\def\modd#1 #2{#1\ \mbox{\rm (mod}\ #2\mbox{\rm )}}

\author{James Haoyu Bai, Joseph Meleshko, Samin Riasat,\footnote{Author's current address:  Department of Electrical Engineering and Computer Science,
University of Michigan,
2260 Hayward Street,
Ann Arbor, MI 48109-2121, USA.} \ and Jeffrey Shallit\footnote{Research supported in part by NSERC grant 2018-04118.}\\
School of Computer Science\\
University of Waterloo\\
Waterloo, ON  N2L 3G1 \\
Canada\\
{\tt \{jhbai,jmeleshko,sriasat,shallit\}@uwaterloo.ca}}

\title{Quotients of Palindromic and Antipalindromic Numbers}

\maketitle

\begin{abstract}
A natural number $N$ is said to be {\it palindromic\/} if its binary representation reads the same forwards and backwards.   In this paper we study the quotients of two palindromic numbers and answer some basic questions about the resulting sets of integers and rational numbers.   For example, we show that the following problem is algorithmically decidable:  given an integer $N$, determine if we can write $N = A/B$ for palindromic numbers $A$ and $B$.   Given that
$N$ is representable, we find a bound on the size of the numerator of the smallest representation.
We prove that the set of unrepresentable integers has positive density in $\Enn$.
We also obtain similar results for quotients
of antipalindromic numbers (those for which the first half of the 
binary representation is the reverse complement of the second half).  We also
provide examples, numerical data, and a number of intriguing conjectures and open problems.
\end{abstract}

\section{Introduction}

Let $\Enn = \{0,1,2,\ldots \}$ denote the natural numbers, and
let $P, Q \subseteq \Enn$ be two given subsets.  
Define the {\it quotient set}
$$ P/Q = \{ p/q \suchthat p \in P, \ q \in Q - \{ 0 \} \}.$$
In the special case where $P = Q$, the set $P/P$ is also known as a 
{\it ratio set\/} in the literature \cite{Bukor&Erdos&Salat&Toth:1997,Bukor&Toth:2003,Kijonka:2007,Luca&Pomerance&Porubsky:2008,Misik:2002, Misik&Toth:2003,Miska&Toth&Zmija:2022,Salat:1969,Salat:2000,Sierpinski:1988,Strauch&Toth:1998,Toth&Zsilinszky:1995}.
Given $P$ and $Q$, six classical problems of number theory are as follows:
\begin{enumerate}
\item What is the topological closure of $P/Q$ in $\Preals$?  In particular,
is $P/Q$ dense in the positive reals $\Preals$?  

\item Consider the following computational problem:  given an integer $N$, is $N \in P/Q$?   Is it algorithmically
decidable? Efficiently decidable?

\item Suppose $N \in P/Q$.  What are good upper and lower bounds on the size of the
smallest representation $N = A/B$ for 
$A \in P$, $B \in Q$?

\item  What are the integers in $P/Q$?
Are there infinitely many?
Are there infinitely many integers not
so representable?   What are the lower and upper densities of
representable and unrepresentable integers in $\Enn$?  
(The lower density of a set $S \subseteq \Enn$ is
$\liminf_{n\rightarrow \infty} {1 \over n} |S \, \cap \, \{1,2, \ldots, n \}|$
and the upper density is
$\limsup_{n\rightarrow \infty} {1 \over n} |S \, \cap \, \{1,2, \ldots, n \}|$.)

\item Given that an integer $N$ belongs to $P/Q$, how many such
representations are there?

\item What are the rational numbers in $P/Q$?

\end{enumerate}

These are, in general, very difficult questions to answer; for some sets $P$, $Q$,
we can even prove that some variations are undecidable \cite[Thm.~5]{Endrullis&Shallit&Smith:2017}.
Let us look at some examples of each of these problems in the literature.

\subsection{Problem 1:  denseness}
\label{subsec11}

As an example of Problem 1, Sierpi\'nski  \cite[p.~165]{Sierpinski:1988}
proved that if $P = Q = {\cal P} = \{ 2,3,5,\ldots \}$,
the set of prime numbers, then $P/Q$ is dense in $\Preals$.  Also see
\cite{Hobby&Silberger:1993,Starni:1995}.  More generally, there is a
criterion originally due to Narkiewicz and {\v{S}al{\'a}t} \cite{Narkiewicz&Salat:1984}, as follows:
\begin{theorem}
Suppose $P = \{ a_1, a_2, \ldots \} \subseteq \Enn$ with
$a_i < a_{i+1}$ for all $i$.  If $\lim_{n \rightarrow \infty}
a_{n+1}/a_n = 1$, then $P/P$ is dense in the positive reals.
\label{nark}
\end{theorem}

As another example, one of the basic steps in the proof of Cobham's famous theorem \cite{Cobham:1969} is the
following observation:  if $P_k := \{ k^i \suchthat i \geq 0 \}$ is the set of powers of an integer $k \geq 2$, then $P_k/P_\ell$ is
dense in $\Preals$ if and only if $k$ and $\ell \geq 2$ are multiplicatively independent.  Also see 
\cite[Prop.~9]{Garcia&Poore&Selhorst-Jones&Simon:2011}.

Let $s_q(n)$ be the sum of the base-$q$ digits of $n$.
Madritsch and Stoll \cite{Madritsch&Stoll:2014} showed that if $P_1$ and $P_2$ are polynomials with integer coefficients, of distinct degrees, such that
$P_1(\Enn), P_2(\Enn) \subseteq \Enn$, then
the sequence of quotients $(s_q(P_1(n))/s_q(P_2(n)))_{n \geq 1}$ is
dense in $\Preals$.

Brown et al.~\cite{Brown&Dairyko&Garcia&Lutz&Someck:2014} proved that if we take $P = Q$ to be the set of integers whose base-$k$ representation starts with $1$, then $P/Q$ is dense in the positive reals if and only if $k \in \{2,3,4 \}$.

Recently, Athreya, Reznick, and Tyson \cite{Athreya&Reznick&Tyson:2019} solved Problem 1 for $P = Q = C$, the Cantor numbers (the natural numbers having no digit ``1'' in their base-3 representation).

 \subsection{Problem 2:  deciding if an integer is representable}

Let $S_1, S_2$ be sets of natural numbers and $L_1, L_2$ the corresponding sets of their canonical base-$b$ representations.   If $L_1$ and $L_2$ are both regular languages (that is, recognized by finite automata), then we can decide whether a given $N \in S_1/S_2$ in $O(N)$ time.   

To see this, 
build an automaton $M$ that accepts, in parallel, the base-$b$ representation of two natural numbers $(A,B)$ if 
$A = BN$, starting with the least significant
digits.  For this we only need $N$ states, to keep
track of the possible carries.     Now use the direct product construction to intersect $M$ with $L_1$ in the first component (corresponding to $A$) and $L_2$ in the second component
(corresponding to $B$), getting an automaton $M'$.
If some final state $M'$ is reachable from the start, then $N$
has a representation; otherwise it does not.   This gives
an algorithm running in $O(N)$ time to decide whether
$N \in S_1/S_2$.   (The implicit constant depends on the size of the finite automata recognizing $L_1$ and $L_2$.)

Of course, is not necessary to construct the entire automaton.  We can use a queue-based algorithm to do breadth-first search on the underlying directed graph of the automaton, implicitly.   If $N$ is representable, we can often find a representation $A/B$ in much less than $O(N)$ time.

\subsection{Problem 3:  size of the smallest representation}

Continuing the example of regular languages, if $N$ has a representation as $A/B$, then
$A = b^{O(N)}$.  This follows from the fact that the automaton $M'$ constructed there has $t$ states, so if $M'$ accepts an input, it must accept an input of length at most $t-1$.   The corresponding integer is then at most
$b^{t-1} - 1$, and $t= O(N)$.

\subsection{Problem 4:  characterizing representable integers}

In 1987, Loxton and van der Poorten \cite{Loxton&vanderPoorten:1987} considered the set $L$ of integers that can be represented in base $4$ using just the digits $0, 1, $ and $-1$.  They showed that every odd integer can be represented as the quotient of two elements of $L$.

Recall the definition of the Cantor numbers $C$ from Section~\ref{subsec11}.
The problem of completely characterizing the ratio set
$V = \Enn \, \cap \, C/C$ was proposed
by Richard Guy \cite[Section F31]{Guy:2004} and is still 
unsolved.     Let $$D = \{ N \suchthat \exists i \geq 1 
\text{ such that } N \equiv \modd{2\cdot 3^{i-1}} {3^i} \} = 
\{ 2, 5, 6, 8, 11, 14, 15, 17, 18, 20,  \ldots \}.$$
By considering the numerator and denominator modulo $3^i$,
it is easy to see that if $N \in D$, then
$N \not\in V$.   Let
$$E = \Enn \, \cap \, \bigcup_{i \geq 0}\ [(3/2)\cdot 3^i, 2 \cdot 3^i]
= \{ 2, 5, 6, 14, 15, 16, 17, 18, 41, 42, 43, 44, 45, 46, 47, \ldots \}.$$
By considering the first few bits in the base-$3$ representation of
numerator and denominator (or using the results in 
\cite{Athreya&Reznick&Tyson:2019}),  it is easy to see that
if $N \in E$, then $N \not\in V$.    It is tempting
to conjecture that $V = \Enn - (D \, \cup \, E)$, but this is false.
Using the algorithm given above for Problem 3,
Sajed Haque and the fourth author of this article found the
following examples of integers in
$F := \Enn - (V \, \cup \, D \, \cup \, E)$:
$$\{ 529, 592, 601, 616, 5368, 50281, 4072741, 4074361,
4088941, 4245688 \}.$$
We do not know if there are infinitely many such examples.  It seems at least possible that numbers of the form $621\cdot 3^{4k} - 20$ might all belong to $F$.

A related conjecture was made by Selfridge and Lacampagne
\cite[\S 7]{Loxton&vanderPoorten:1987}.
   If we let $B = \{ 1, 2, 4, 5, 7, 11, 13, 14, 16, 20, 22, \ldots \}$ be the set of natural numbers having no $0$ in their
balanced ternary representation, then they conjectured that every
$n \not\equiv \modd{0} {3}$ belongs to $B/B$.  However, we found the
counterexamples
$$ \{ 247, 277, 967, 977, 1211, 1219, 1895, 1937, 1951, 1961, 2183, 2191, 2911, 2921,
       3029, 3641, 3649\},$$
the first of which was also found by Coppersmith \cite[Section F31]{Guy:2004}.
It seems likely that there are infinitely many such counterexamples, but we have no proof.

For a different example, let
$U = \{ 2^{k+1} + i \suchthat 1 \leq i \leq 2^{k-1} \}$.   
\v{S}al\'at \cite{Salat:1969} observed that $U/U$ has lower density $1/4$ and
upper density $2/5$.

\subsection{Problem 5:  counting number of representations}
Consider $S = \{ 1,2,4,5,8,9,10, \ldots \}$, the set of integers that can
be written as the sum of two squares of natural numbers.  Then it follows
from Fermat's classical two-square theorem that $S/S = S$.  Hence every $N \in S/S$ has infinitely many representations of the
form $N = A/B$ with $A,B \in S$.

\subsection{Problem 6:  which rationals are representable?}

As an example, Sierpi\'nski observed \cite[p.~254]{Sierpinski:1988} that if we take
$P = Q = \{ \varphi(n) \suchthat n \geq 1 \}$, the range of Euler's totient function, then
$P/Q$ contains every positive rational number.

On the other hand, it is a nice exercise in elementary number theory
to show that every non-negative rational number belongs to $\Enn/T$, 
where $T = \{ (2^i - 1)2^j \suchthat i \geq 1, j \geq 0 \}$.   See
 \cite[Example 7]{Rowland&Shallit:2015}.
 
 Define ${\cal E} = \{ 0,3,5,6,\ldots \} = \{ n \in \Enn \suchthat t(n) = 0 \}$ and  ${\cal O} = \{1,2,4,7,\ldots \} = \{ n \in \Enn \suchthat t(n) = 1 \}$, where $t$ is the Thue-Morse
 sequence.  
Stoll \cite{Stoll:2015} showed that for odd natural numbers $p > q$ there are integers
$n_1, n_2 < p$ such that $t(n_1p), t(n_1q) \in 
{\cal E}$,
and $t(n_2p), t(n_2q) \in {\cal O}$.
Since $t(2n) = t(n)$,
we immediately get that
${\cal E}/{\cal E}$ and ${\cal O}/{\cal O}$ both contain all positive rational numbers.

\section{Palindromic and antipalindromic numbers}

Now that we have motivated the study of the properties of $P/Q$ for sets $P, Q$, we turn to considering
Problems 1--6 above for $P = Q$ in a novel context:    the palindromic and antipalindromic numbers.  These two
classes have previously been studied by number theorists; see, e.g.,
\cite{Banks:2016,Cilleruelo&Luca&Baxter:2017,Rajasekaran&Shallit&Smith:2020}. 

We say that a natural number is {\it palindromic} if its base-$b$
representation is a palindrome (reads the same forwards and backwards).  For base $2$, the palindromic numbers ${\tt PAL} = \{ 1,3,5,7,9,15,17,\ldots\}$ form sequence \seqnum{A006995} in the {\it On-Line Encyclopedia of Integer Sequences} (OEIS).

Analogously, we say that a natural number
is {\it antipalindromic} if its base-$2$
representation is of even length, and the
second half is the reverse complement of the
first half.  For example, $52$ (which is
$110100$ in binary) is antipalindromic.  The
antipalindromic numbers ${\tt APAL} = \{ 2, 10, 12, 38, 42, 52, 56, \ldots \}$
form sequence \seqnum{A035928} in the OEIS. 
This can be generalized to base $b$ by demanding that if $a$ is a digit
in the first half of a number's representation, and $a'$ is the corresponding digit in the reverse of the
second half, then $a+a' = b-1$.

As it turns out, the study of the palindromic and antipalindromic numbers is particularly amenable to tools from automata theory and formal languages.
These tools have previously been used to solve other kinds of 
number theory problems (see, e.g., \cite{Rajasekaran&Shallit&Smith:2020}).

Our principal interest in this paper is base $2$, although nearly everything we say can be generalized to other bases.    We let $\qpal = \Enn \, \cap \, {\tt PAL}/{\tt PAL}$, the integers representable as quotients of palindromic numbers, and
$\qapal = \Enn \, \cap \, {\tt APAL}/{\tt APAL}$, the integers representable as quotients of antipalindromic numbers.

Throughout the paper we must distinguish between an integer and its
base-$k$ representation.   For $n \geq 1$, define $(n)_k$ to be the string of digits
representing $n$ in base $k$, starting with the most significant digit,
which must be nonzero.   If $w$ is a string of digits over the alphabet
$\Sigma_k = \{ 0, 1, \ldots, k-1 \}$, then by $[w]_k$ we mean the
integer represented by $w$ in base $k$.  Thus, for example,
$(43)_2 = 101011$ and $[101011]_2 = 43$.

For a string $x$, by $x^n$ we mean the string
$\overbrace{xx\cdots x}^n$.   In some cases (for example, an equality such as $1^4 = 1111$) there
could be ambiguity between this notation and the ordinary notation for powers
of integers, but the context should make it clear which interpretation is
meant.

We use the notation $\overline{a}$ to denote the binary complement of
the bit $a$:   $\overline{0} = 1$ and $\overline{1} = 0$.  This can be
extended to strings $w$ in the obvious way.   Another extension is that
if we are working over base $b$, then we can define $\overline{a} = b-1-a$.
Here the choice of $b$ should be clear from the context.

The {\it Hamming distance\/} $h(x,w)$ between two identical-length strings,
$x$ and $w$, is defined to be the number of positions on which $x$ and $w$ differ.

\subsection{Denseness}

\begin{theorem}
The ratio set $\pal/\pal$ is dense in the positive reals.
\label{one}
\end{theorem}

\begin{proof}
Let $\alpha > 0$ be a real number that we want to approximate as
the quotient of two palindromic natural numbers.  Without loss
of generality, we can assume $\alpha \leq 1$ (otherwise, we represent
the reciprocal $1/\alpha$).  Let $k \geq 0$ be an integer
such that ${1 \over 2} < 2^k \alpha \leq 1$, and set
$\beta = 2^k \alpha$.

We now approximate $\beta$ by forming a palindrome from the
first $n$ bits of the binary expansion of $\beta$ (duplicating
the bits, then reversing and appending them), and dividing by 
the palindromic number $B = 2^{2n+k} + 1$.  More formally, let
$\gamma = \lfloor 2^n \beta \rfloor$, and
define $A = [(\gamma)_2 \, (\gamma)_2^R]_2$.  Then $A/B \approx \alpha$, and it remains to see how good this approximation is.

Clearly $\gamma \leq A/2^n < \gamma+1$.  Therefore
$$ 2^n \beta -1 < \lfloor 2^n \beta \rfloor = \gamma \leq A/2^n < \gamma+ 1 = \lfloor 2^n \beta \rfloor + 1 \leq 2^n \beta + 1.$$
Multiplying through by $2^n/B$ gives
$$ {{2^{2n}\beta - 2^n} \over{2^{2n+k} + 1}} < {A \over B} <
{{2^{2n} \beta + 2^n} \over {2^{2n+k} + 1}} ,$$
or equivalently,
\begin{equation}
 {{ \beta - 2^{-n}} \over {2^k + 2^{-2n}}} < {A \over B } < 
{{\beta + 2^{-n}} \over {2^k + 2^{-2n}}} < {\beta \over 2^k} + 2^{-n-k} .
\label{app}
\end{equation}
Now
\begin{align*}
{{\beta - 2^{-n}} \over {2^k + 2^{-2n}}} &= {{\beta - 2^{-n}} \over
{2^k}} \left( {1 \over {1 + 2^{-2n-k}}} \right)  \\
&> {{\beta - 2^{-n} - 2^{-2n-k} \beta + 2^{-3n-k}} \over
{2^k}} \\
&> {{\beta - 2^{-n} - 2^{-2n-k}} \over {2^k}},
\end{align*}
where we have used the fact that $\beta < 1$ and the estimate
$1/(1+x) > 1-x$.  Substituting in Eq.~\eqref{app}, we see that
$$ \alpha - 2^{-n-k} - 2^{-2n-2k} < {A \over B} < \alpha + 2^{-n-k}.$$
Hence, as $n \rightarrow \infty$, the quotient of palindromes
$A/B$ gets as close as desired to $\alpha$.
\end{proof}

\begin{remark}
We could have also proved Theorem~\ref{one} using the criterion in
Theorem~\ref{nark}.
\end{remark}

\subsection{Testing if $N$ is the quotient of palindromic numbers}

We now turn to the question of deciding,
given a natural number $N$, whether there exist palindromes
$A, B$ such that $N = A/B$.    Since a positive number must be odd for its base-$2$
representation to be a palindrome, it is clear that
only odd integers are representable.

The set $\qpal$
$$ 1,3,5,7,9,11,13,15,17,19,21,27,31,33,39,\ldots$$
of positive integers having such a representation
is sequence \seqnum{A305468} in the OEIS.

The sequence  
$$ 23,25,29,35,37,41,47,49,59, \ldots $$
of odd positive integers having no representation
as the quotient of palindromic numbers
is sequence \seqnum{A305469} in the OEIS.

Evidently, if there exist such $A, B$ we can find one through a brute-force search, so for the moment we focus on how we might establish that there is no such solution.   We describe three algorithms:  a heuristic algorithm that does not always terminate; a rigorous algorithm based on context-free languages; and
finally, a fast rigorous algorithm based on deterministic finite automata.

\subsubsection{A heuristic algorithm}
\label{heuristic}

There is a fast and relatively simple heuristic method to solve this problem that works in many cases, but
is not guaranteed to terminate.  If it does terminate, the
answer it gives is guaranteed to be correct.   We describe it now.
Suppose we are considering a candidate $T$ for the first $k$ bits of $B$.
Since $A = BN$, these
$k$ bits of $B$ determine all the possibilities for the first $k$ bits of $A$.

On the other hand, the first $k$ bits of $B$
determine 
the last $k$ bits of $B$.
By considering the equation $A = BN$ modulo $2^k$, the last $k$ bits of $A$ are also completely determined.  Hence the first $k$ bits of $A$ are completely determined, and must match one of the possibilities in the preceding paragraph.  If they do not, we have
ruled out $T$ as the possibility for the first
$k$ bits of $A$.  

We now do a breadth-first search over the tree of possible prefixes of $B$.  The hope is that we either find a solution, or are able to rule out all possibilities for the solution of $A/B = N$.
This will be the case if the following heuristic principle holds:
\begin{hprinciple}
If there is no solution in palindromes $A, B$ to the equation $A/B = N$, then this fact can be proved by examining all possible $k$-bit prefixes of $B$ for some fixed integer $k$ (which might depend on $N$) .
\end{hprinciple}

We illustrate the basic idea for $N = 35$.  Suppose $A,B$ are palindromes with $A/B = 35$.
Then the first three bits of $B$ are either $100, 101, 110, 111$. 

Let's assume the first three bits
are $100$.  Then, since $A = 35B$, we see
that the first three bits of $A$ are
either $100$ or $101$.  On the other hand
the last three bits of $B$ are $001$, so
from $A = 35B$ we see the last three bits
of $A$ are $011$.  So the first three bits
of $A$ are $110$, contradicting what we
found earlier.

Similar contradictions occur for the other
three possibilities, so we have proved that there is no solution in palindromic numbers
to the equation $A/B = 35$.

Using our heuristic algorithm, we were able to determine the representability of all odd $N \leq 2000$.   The data for $N \leq 239$ is given in
Table~\ref{tabone}.  Here $k$ denotes the length of the largest bit strings
that were needed to prove that $N = A/B$ has no solutions in palindromic
numbers.  
\begin{table}[H]
\begin{center}
\begin{tabular}{c|c|c|c||c|c|c|c||c|c|c|c}
$N$ & $A$ & $B$ & $k$ & $N$ & $A$ & $B$ & $k$ & $N$ & $A$ & $B$ & $k$  \\
\hline
1 & 1 & 1 &                       & 81 & --- & --- & 3                & 161 & --- & --- & 3              \\
3 & 3 & 1 &                       & 83 & 3735 & 45 &                  & 163 & 7335 & 45 &                \\
5 & 5 & 1 &                       & 85 & 85 & 1 &                     & 165 & 165 & 1 &                  \\
7 & 7 & 1 &                       & 87 & --- & --- & 4                & 167 & --- & --- & 5              \\
9 & 9 & 1 &                       & 89 & --- & --- & 3                & 169 & --- & --- & 3              \\
11 & 33 & 3 &                     & 91 & 273 & 3 &                    & 171 & 513 & 3 &                  \\
13 & 65 & 5 &                     & 93 & 93 & 1 &                     & 173 & 5709 & 33 &                \\
15 & 15 & 1 &                     & 95 & 2565 & 27 &                  & 175 & --- & --- & 8              \\
17 & 17 & 1 &                     & 97 & --- & --- & 3                & 177 & --- & --- & 3              \\
19 & 513 & 27 &                   & 99 & 99 & 1 &                     & 179 & 11277 & 63 &               \\
21 & 21 & 1 &                     & 101 & --- & --- & 5               & 181 & 16833 & 93 &               \\
23 & --- & --- & 4                & 103 & --- & --- & 7               & 183 & --- & --- & 4              \\
25 & --- & --- & 3                & 105 & --- & --- & 6               & 185 & --- & --- & 3              \\
27 & 27 & 1 &                     & 107 & 107 & 1 &                   & 187 & 561 & 3 &                  \\
29 & --- & --- & 8                & 109 & 2289 & 21 &                 & 189 & 189 & 1 &                  \\
31 & 31 & 1 &                     & 111 & --- & --- & 6               & 191 & 29223 & 153 &              \\
33 & 33 & 1 &                     & 113 & --- & --- & 4               & 193 & --- & --- & 3              \\
35 & --- & --- & 3                & 115 & --- & --- & 3               & 195 & 195 & 1 &                  \\
37 & --- & --- & 6                & 117 & 585 & 5 &                   & 197 & --- & --- & 6              \\
39 & 195 & 5 &                    & 119 & 119 & 1 &                   & 199 & --- & --- & 9              \\
41 & --- & --- & 3                & 121 & 11253 & 93 &                & 201 & --- & --- & 3              \\
43 & 129 & 3 &                    & 123 & --- & --- & 3               & 203 & 1421 & 7 &                 \\
45 & 45 & 1 &                     & 125 & --- & --- & 5               & 205 & 1025 & 5 &                 \\
47 & --- & --- & 6                & 127 & 127 & 1 &                   & 207 & --- & --- & 6              \\
49 & --- & --- & 3                & 129 & 129 & 1 &                   & 209 & --- & --- & 4              \\
51 & 51 & 1 &                     & 131 & --- & --- & 3               & 211 & 633 & 3 &                  \\
53 & 3339 & 63 &                  & 133 & 3591 & 27 &                 & 213 & 54315 & 255 &              \\
55 & 165 & 3 &                    & 135 & --- & --- & 4               & 215 & 645 & 3 &                  \\
57 & 513 & 9 &                    & 137 & --- & --- & 8               & 217 & --- & --- & 8              \\
59 & --- & --- & 3                & 139 & --- & --- & 3               & 219 & 219 & 1 &                  \\
61 & 427 & 7 &                    & 141 & --- & --- & 6               & 221 & 1105 & 5 &                 \\
63 & 63 & 1 &                     & 143 & 2145 & 15 &                 & 223 & 2965677 & 13299 &          \\
65 & 65 & 1 &                     & 145 & --- & --- & 4               & 225 & --- & --- & 4              \\
67 & --- & --- & 3                & 147 & --- & --- & 6               & 227 & --- & --- & 3              \\
69 & --- & --- & 7                & 149 & 5887437 & 39513 &           & 229 & 3435 & 15 &                \\
71 & 54315 & 765 &                & 151 & 1057 & 7 &                  & 231 & 231 & 1 &                  \\
73 & 73 & 1 &                     & 153 & 153 & 1 &                   & 233 & 59415 & 255 &              \\
75 & --- & --- & 4                & 155 & --- & --- & 7               & 235 & --- & --- & 3              \\
77 & 231 & 3 &                    & 157 & 471 & 3 &                   & 237 & --- & --- & 6              \\
79 & 888987 & 11253 &             & 159 & 3339 & 21 &                 & 239 & 717 & 3 &                 
\end{tabular}
\end{center}
\caption{Results of the heuristic algorithm for odd $N\leq 239$}
\label{tabone}
\end{table}

Unfortunately, the heuristic principle does not hold in all cases.     We found six examples less than
$20000$ for which a failure to terminate occurs.   They are summarized
in Table~\ref{exceptions}.  For each entry we have
$N \cdot [ r s^n (s^R)^n r^R ]_2 =
[t u^{n-i} v w v^R (u^R)^{n-i} t^R]_2$ for $n \geq 2$,
and furthermore $\pald(w) = d$.   Here $\pald(w) = h(w,w^R)$, the Hamming distance between $w$ and $w^R$.
For these numbers there is an infinite sequence $(f(n))$ of palindromic numbers whose product with $N$ is ``almost" palindromic, and furthermore the first bit position where this product differs from being a palindrome is located arbitrarily far in (and hence will
never be detected by an algorithm that focuses only on fixed-size prefixes).

\begin{table}[H]
\begin{center}
\resizebox{\textwidth}{!}{%
\begin{tabular}{ccccccccc}
$N$ & $r$ & $s$ & $t$ & $u$ & $v$ & $w$ & $i$ & $d$ \\
\hline
2551 & $\epsilon$& 10100010000 & 1100100 & 11110001011 
& 111 &  0010110001011 & 1 & 12 \\
14765 & $\epsilon$ & 111011110110 & 
1101011111000 & 110000010111 &  1100000101101 & 1011010110 & 2 & 8\\
15247 & $\epsilon$ & 11001101110011001000 & 
10111111100 & 00100011001101110011& 0010001 & 011100000011000101 & 1 & 10 \\
17093 & $\epsilon$ & 110111001000 & 11100110000 & 110010001101 & $\epsilon$ & 0110000000010101 & 1 & 6 \\
19277 & 11 & 0000100011100111110111000110 &
1110010010000101 & 1001101101001101100100101100 & 10011011010011 & 11001111100100 & 1 & 8 \\
19831 & $\epsilon$& 11101010111100 & 1000111000110 & 00100000011111 & 0010000001111 & 0111010111111010101 & 2 & 12
\end{tabular}
}
\end{center}
\caption{Some $N$ for which the heuristic principle fails.}
\label{exceptions}
\end{table}

Let's verify the claim for $N = 2551$.
For the given
$r,s,t,u,v,w$ we have
$$[s^n (s^R)^n]_2 = 1296\cdot 2^{11n} \cdot {{2^{11n}- 1} \over {2^{11} - 1}}
+ 69 \cdot {{2^{11n}- 1} \over {2^{11} - 1}} $$
while
\begin{align*}
&[t u^{n-1} v w v^R (u^R)^{n-1} t^R]_2 =
[t]_2 \cdot 2^{|u^{n-1} v w v^R (u^R)^{n-1} t^R|} +
[u]_2 \cdot 2^{|v w v^R (u^R)^{n-1} t^R|} {{2^{(n-1)|u|} - 1} \over {2^{|u|} - 1}}  \\
& +[v]_2 \cdot 2^{|w v^R (u^R)^{n-1} t^R|} +
[w]_2 \cdot 2^{|v^R (u^R)^{n-1} t^R|} +
[v^R]_2 \cdot 2^{|(u^R)^{n-1} t^R|} + 
[u^R]_2 \cdot 2^{|t^R|} {{2^{(n-1)|u^R|} - 1} \over {2^{|u^R|} - 1}} + [t^R] \\
&= 100 \cdot 2^{11(n-1)+3+13+3+11(n-1)+7} +
1931 \cdot 2^{3+13+3+11(n-1)+7} {{2^{11(n-1)} - 1} \over {2^{11} - 1}} + \\
&7 \cdot 2^{13+3+11(n-1) + 7} +
1419 \cdot 2^{3 + 11(n-1) +7} +
7  \cdot 2^{11(n-1) + 7} + 
1679 \cdot 2^7 {{2^{11(n-1)} - 1} \over {2^{11} - 1}} + 19.
\end{align*}
The expression for $[s^n (s^R)^n]_2$ simplifies to
$$ {{1296} \over {2^{11}-1}} \cdot 2^{22n} -
{{1227} \over {2^{11}-1}} \cdot 2^{11n} ,$$
while the expression for $[t u^{n-1} v w v^R (u^R)^{n-1} t^R]_2$ simplifies to 
$$ {{3306096} \over {2^{11}-1}} \cdot 2^{22n} -
{{3130077} \over {2^{11}-1}} \cdot 2^{11n}.$$
It is now easily verified that the second is 2551
times the first.

\begin{conjecture}
There are infinitely many natural numbers $N$ for which the heuristic
algorithm fails to terminate.
\end{conjecture}

\subsubsection{A provable decision procedure}

In contrast to the fast method presented in
Section~\ref{heuristic}, in this section we describe another technique
that provides a provable decision procedure.  This method is
based on formal language theory.

Here is a brief sketch of the idea:  first, given $N$, we construct a pushdown
automaton (PDA) $M_N$ that, on input $A$ and $B$ expressed in binary, and read in parallel, determines if
$A$ and $B$ are both palindromes and if $A = BN$.  Next, we
convert $M_N$ to an equivalent context-free grammar (CFG) $G_N$.
Finally, we use a standard decision procedure for context-free
grammars to decide if $G_N$ generates any string, and if so,
to find the shortest string generated by $G_N$.

However, there are some complications.  While determining if $A$
is palindromic with a PDA is easy, making the same determination
for $A$ and $B$ simultaneously (when they are of different
lengths) is harder.  To align $A$ and $B$ around their center,
we multiply $B$ by $2^k$ for some appropriate power of $2$.  
Thus, instead
of checking whether $A = BN$, we are actually checking
if $2^k A = BN$.  Now there are four
separate cases to examine, depending on the parity of the length
of $(A)_2$ and $(B)_2$.

Our solution consists of five parts:

\begin{itemize}

\item
{\tt ConstructPDA($N$):}  on input a positive integer $N$, constructs four PDAs that accept the base-2 representation of all
$(A,B)$ in parallel such that $A = BN$ and both $A$ and $B$ are palindromes.   This PDA has $O(N^{3/2})$ states, where $O(N)$ states
are used to keep track of the multiplication by $N$, and an additional
multiplicative factor of $O(N^{1/2})$ states required to keep track
of the symbols required to ``line up'' the binary
representation of $A$ with $B$.

\item
{\tt CanonicalPDA($M$):}  on input $M$ returns a new PDA $M'$ that is in
Sipser normal form:  it has at most one final state,
empties the stack before accepting, and each transition either pushes exactly one symbol onto the stack or pops one off.   

\item {\tt PDA-to-Grammar($M$):}  takes a PDA $M$ in Sipser normal form and returns
an equivalent CFG $G$ using the algorithm in 
\cite{Sipser:2013}.    This blows up the number of states by
at most a cubic factor, so the size of the grammar is
$O(N^{9/2})$.

\item {\tt Remove-Useless-Symbols($G$):}  takes a CFG $G$ and removes useless
symbols (both variables and terminals) following the algorithm in
Hopcroft and Ullman \cite{Hopcroft&Ullman:1979}.  If nothing is left, we know $L(G)$ is the empty set.

\item {\tt Shortest-String-Generated($G$):}  given that the CFG $G$ generates at least
one string, this routine returns the shortest string (or perhaps strings) generated by $G$, 
using dynamic programming. 

\end{itemize}

Using these ideas we were able to prove
\begin{theorem}
There exists an algorithm to determine if $N$ can be written as the quotient of palindromic numbers that runs in $O(N^{9/2})$ time.
\end{theorem}








This method was programmed up by the first author in 2019, and with it we
were able to determine the solvability of $N = A/B$ in palindromes
for all odd numbers $\leq 600$.   Unfortunately, it was too slow to resolve the cases we were interested in (such as $N = 2551$, which the heuristic algorithm could not solve), so we turned to another method described in the next section.

\subsubsection{A different provable decision procedure based on finite automata}
\label{decision2}

We developed another method that is based on finite automata (instead of pushdown automata).   Of course, finite automata cannot recognize palindromes, so
we have to be a bit more clever.

Let $N$ and $k$ be integers.   The case of the representability
of $N = A/B$ with $A, B$ palindromes in base $k$ is easy to decide
in the cases where $N < k$ or $k \mid N$, so we assume neither of
these holds.

We construct a nondeterministic finite automaton $M_{N, k}$ to check whether $N$ can be expressed as the quotient of palindromic numbers in base $k$.
This automaton accepts certain pairs of strings $a$ and $b$, from which we derive integers $A$ and $B$, where $(A)_k$ and $(B)_k$ are palindromes and $A/B = N$.
This is accomplished by interpreting the input each $a$ and $b$ as half of a palindrome $(A)_k$ and $(B)_k$, respectively, and then verifying the equation $A = BN$.
The automaton verifies the equation from both the left-hand and right-hand halves of the digits of $(A)_k$ and $(B)_k$ simultaneously.
From the size of the constructed $M_{N, k}$ we can also obtain a bound on the maximum size of $A$. 

\subsubsection*{Verifying a multiplication with a system of equations}
To verify the equation $A = B \cdot N$ we will compare $N \cdot B$ to $A$ digit by digit.
Let $(A)_k = A_i A_{i-1} \cdots A_1$ and $(B)_k = B_j B_{j - 1} \cdots B_1$.
We begin by checking $$A_1 = (N \cdot B_1) \bmod k.$$
This leaves a carry to contribute to the next equation $$c_1 = \frac{N \cdot B_1 - A_1}{k}.$$
We call these $c_\ell$ values the \textit{carries}.
We then subsequently verify each equation
$$A_\ell = (N \cdot B_\ell + c_{\ell - 1}) \bmod k$$
for $\ell \in \{2, 3, \ldots, \abs{(A)_k}\}$.
When $\ell > \abs{(B)_k}$ we continue with $B_\ell = 0$.
At each step we get a new equation
$$c_\ell = \frac{N \cdot B_\ell + c_{\ell - 1} - A_\ell}{k}$$
for the next step.
If at the end of the process we have that $c_i = 0$, then all the equations are valid and indeed $A = N \cdot B$.

We can also obtain a bound on the size of $c_\ell$. This contributes to the bound on the size of $M_{N, k}$.
We have $$c_\ell \leq \frac{(k-1) \cdot N + c_{\ell - 1} - 0}{k}.$$
Since the carry starts at 0, $c_1$ includes $c_0 = 0$ so we have that $$c_1 \leq \frac{k-1}{k} \cdot N < N.$$
We can then assume for the sake of induction that $c_{\ell - 1} < N$ and get that
\begin{align*}
c_\ell &= \frac{N \cdot B_\ell + c_{\ell - 1} - A_\ell}{k}\\
&\leq \frac{(k-1) \cdot N + c_{\ell-1} - 0}{k}\\
&< \frac{(k-1) \cdot N + N}{k}\\
&< \frac{k \cdot N - N + N}{k}\\
&< \frac{k \cdot N}{k}\\
&< N,
\end{align*}
as $A_\ell \geq 0$ and $B_\ell \leq k-1$.
We can use the fact that $c_\ell$ is bounded by $N$ to constrain the states we have to consider in $M_{N, k}$.
Any state with a left carry larger than or equal to $N$ cannot lead to an accepting state, so we can safely omit it.

The automaton $M_{N,k}$ simultaneously checks the equations starting with $\ell = 1$ in ascending order and the equations starting with $\ell = i$ in descending order.
The ascending equations have a carry computed as previously described.
The descending equations start with the assumption that $c_i = 0$ and compute the required preceding carry value that would result in the equation being satisfied.
We compute $c_{\ell-1}$ from the relation
$$c_{\ell-1} = k \cdot c_\ell - N \cdot B_\ell + A_\ell.$$
The states of the automaton keep track of the value of the largest index carry computed from the right and the smallest index carry computed from the left.
An accepting state is one where the top and bottom carries are equal. This implies that each equation is satisfied from $\ell = 1$ up to $\ell = i$.

\subsubsection*{Palindromes as input to an automaton}

There are two main challenges regarding the input specification when trying to design an automaton that verifies an equation and ensures that the inputs are palindromes.
The first challenge is that it is impossible to recognize a palindrome with a finite automaton. 
To remedy this issue we take, as input, {\it half\/} of a palindrome and implicitly determine the other half.
A naive approach is to interpret the input pair $\langle a,b \rangle$ as referring to the equation $[aa^R]_k = [bb^R]_k \cdot N$.

This means all even-length palindromes have an associated string that is a valid input to our automaton.
However, this does not cover the case of odd-length palindromes.
Therefore, on input $\langle a, b \rangle$, the automaton $M_{N, k}$ simultaneously checks each equation $[a\sigma_a a^R]_k = [b\sigma_b b^R]_k \cdot N$ where $\sigma_a, \sigma_b \in \{\epsilon\} \cup \Sigma_k$. If any of the equations are valid, then the automaton accepts the input.

The second challenge is that the strings $(A)_k$ and $(B)_k$ have, in general, different lengths.
Furthermore, the difference in length between them could be either the floor or the ceiling of $\log_k N$.
To accommodate both possibilities, $M_{N, k}$ begins by nondeterministically guessing the difference in length between $(A)_k$ and $(B)_k$.
Since $\abs{(A)_k} > \abs{(B)_k}$, it follows that $\abs{a} \geq \abs{b}$.
If $\abs{(N)_k} = 2$, then it is possible that there is a satisfying $a$ and $b$ where $\abs{a} = \abs{b}$.
However, in general we need to pad $b$ to provide it as input to the automaton simultaneously with $a$.
We use $X$ as a padding character to indicate the end of input for $b$.
We format the input $a$ and $b$ as $\langle a, b \rangle \in \big(\Sigma_k \times (\Sigma_k \cup \{X\}) \big)^*$.
The automaton $M_{N, k}$ rejects any input not of the form $\langle a, b \rangle = xy$ where $x \in (\Sigma_k \times \Sigma_k)^*$ and $y \in (\Sigma_k \times \{X\})^*$.
Additionally, the automaton rejects an input that begins with either $a[1]$ or $b[1]$ being zero.
This would result in a palindrome representing a number that would not be a palindrome in canonical representation.

\subsubsection*{Checking equations for the first component of the input}

This section describes the states that read the component of the input composed of symbols in $\Sigma_k \times \Sigma_k$.
The automaton is able to directly check the equations and compute the carries for the right-hand side, since each input from $\Sigma_k \times \Sigma_k$ contains all the information for one set of equations.
The first symbol of $\langle a,b \rangle$ is $(a[1], b[1])$.
Since $A_1 = a[1]$ and $B_1 = b[1]$, $(a[1], b[1])$ has all the information required for the equations
$$A_1 = (N \cdot B_1) \bmod k$$ and $$c_1 = \frac{N \cdot B_1 - A_1}{k}.$$
Afterwards, the automaton saves the carry for the next equation.
On receiving each input $(a[\ell], b[\ell]) = (A_\ell, B_\ell)$, the automaton is able to check the equation
$$A_\ell = (N \cdot B_\ell + c_{\ell-1}) \bmod k$$ and compute $$c_\ell = \frac{N \cdot B_\ell + c_{\ell-1} - A_\ell}{k}.$$
Therefore, the only information that $M_{N, k}$ must
preserve between states in order to verify these equations is the current value of the carry.
We call this saved value the \textit{right carry}.

The left-hand side requires more careful handling.
The automaton does not verify the equations on the left side, instead it asserts that they will be valid and computes the carry required from the right to satisfy the current step.
Since $(A)_k$ is a palindrome in canonical notation and there is a difference in length between it and $(B)_k$, we must have $A_i = a[1]$ and $B_i = 0$.
Using $c_i = 0$ from the assumption of satisfaction, $M_{N, k}$ computes $c_{i-1}$ with the equation
$$c_{i-1} = k \cdot c_i - N \cdot B_i + A_i = k \cdot 0 - N \cdot 0 + A_i = A_i.$$
The automaton preserves the carry for the next equation and we call this saved value the \textit{left carry}.
The automaton proceeds with calculating $c_{\ell-1}$ with $A_{\ell} = a[i-\ell+1]$, $B_\ell = 0$, and $c_\ell$ from the previous step with the equation
$$c_{\ell-1} = k \cdot c_\ell - N \cdot B_\ell + A_\ell = k \cdot c_\ell - N \cdot 0 + A_\ell = k \cdot c_\ell + A_\ell.$$
The equation using $a[\ell]$ to compute the left carry is computed concurrently with the corresponding equation on the right-hand side to compute the right carry that also uses $a[\ell]$. (This event is upon reading $(a[\ell], b[\ell])$).
We call this the \textit{loading phase}.

The left-hand side continues as described until trying to compute $c_{j-1}$.
At this step, $M_{N, k}$ needs $B_j = b[1]$ along with $a[i-j+1]$ to compute
$$c_{j-1} = k \cdot c_j - N \cdot B_j + A_j.$$
In order to compute an equation requiring information contained in different input symbols the automaton saves some additional information beyond the two carries. After the first input symbol $(a[1], b[1])$ is read and the right and left carries are computed, $M_{N, k}$ preserves $b[1]$.
The automaton keeps $b[1]$ until it has to compute $c_{j-1}$ and at that point discards it as it won't need it for any other calculations.
Similarly, to compute $c_{j-2}$, $M_{N, k}$ needs $b[2]$ which gets preserved after reading $(a[2], b[2])$ and discarded at the step computing $c_{j-2}$.
Each time that an $(a[\ell], b[\ell])$ is read the $b[\ell]$ must be saved for a later equation.
This process of using, discarding, and then subsequently replacing a saved symbol continues while the input symbols are of the form $(a[\ell], b[\ell])$. (This means this phase continues until we've seen all of $b$.)
We call the section of computation where $M_{N, k}$ consumes and discards saved symbols while still saving new ones the \textit{shifting phase}.

The number of symbols of $b$ that need to be saved is the difference in length between $(A)_k$ and $(B)_k$.
As stated previously, this difference can vary between the floor and ceiling of $\log_k N$.
To accommodate both possibilities, $M_{N, k}$ nondeterministically assumes that the difference is a fixed value $m$ and the loading phase saves that many symbols of $b$ before starting to consume and replace them in the shifting phase.
We call the currently saved section of $b$ the \textit{queue of saved symbols}.

Each state of $M_{N, k}$ is therefore identified by the $\leq m$ symbols saved, the integer $m$ itself, the left and right carries, and what phase the automaton is in.
The automaton also has a special start state, with an $\epsilon$ transition to the two states with no symbols saved, left and right carries set to 0, and each possibility for $m$.
For all other states, the automaton has a transition to the resulting state when the equations are checked, carries updated, and saved symbols updated as per the input on the transition and the current status as given by the original state.
If the associated equation on the right-hand side isn't verified or it results in a carry larger than $N$, then the transition is omitted.
The loading stage is characterized by having less than $m$ saved symbols and the shifting phase having exactly $m$ saved symbols that it cycles through.

\subsubsection*{Checking equations for the second component of the input}

Once $M_{N, k}$ has seen all of the input $b$, the input changes to being of the form $(a[\ell], X)$.
We call this final section of processing the \textit{unloading phase}.
Any transition with an input of the form $(a[\ell], X)$ pushes the automaton directly into the unloading phase
This can lead to not having $m$ saved symbols in the queue of saved symbols despite having read all of $b$.
If this occurs, $M_{N, k}$ implicitly pads the front of the queue of saved symbols with enough zeroes to have $m$ saved symbols.
At this point the automaton has all the digits of $(B)_k$ (except $\sigma_b$), and has yet to examine the middle section of $(A)_k$ that corresponds to the remainder of $a$.
At this point, when the automaton reads a symbol ($a[\ell]$, X), it represents $a[\ell]$ which is both the leftmost and rightmost digit of the unexamined middle of $(A)_k$.
This middle section lines up with the queue of saved symbols to supply the $b$ symbols that are no longer coming from the input.

The automaton must now contend with the possibility that $(B)_k$ has odd length and there may be a symbol of $(B)_k$ not given in $b$.
It nondeterministically decides what the central symbol $\sigma_b \in \Sigma_k \cup \{X\}$ is for $(B)_k$.
If $\sigma_b \neq \epsilon$, then it proceeds using the input $a[\ell]$ as the symbol from $(A)_k$ for both the left and right side equations.
The automaton uses the chosen symbol as the $(B)_k$ symbol for the right-hand side equations since we have already processed the entire right-hand side of $(B)_k$.
The left-hand side, as usual, pops the first in symbol in the queue of saved symbols but doesn't add anything else to the queue since there is no new $b[\ell]$.
The left and right carries are updated as usual and the automaton continues with a reduced queue of saved symbols.
If $M_{N, k}$ decides that $\sigma_b = \epsilon$, then it skips the step described in this paragraph and proceeds directly with the subsequent steps.

At this point, $M_{N, k}$ consumes both ends of the queue of saved symbols to compute the usual equations for the left-hand and right-hand sides with the input $a[\ell]$.
This proceeds, consuming two symbols of the queue of saved symbols each time.
Once the automaton has less than 2 symbols left, there are two remaining cases.
If there are 0 saved symbols remaining and the left and right carries are equal, then the automaton accepts the input.
In this case, the entire series of equations are satisfied and the input represents a valid $A/B = N$ with $(A)_k$ and $(B)_k$ palindromes.
Alternatively, if the automaton has one saved symbol left, then this is the case where $(A)_k$ has an odd number of symbols.
If there is an assignment for the middle symbol $\sigma_a$ that results in the carries being equal, then the automaton accepts the input.

\subsubsection*{Algorithm implementation}

We implemented an algorithm for computing the desired $A$ and $B$ for a given $N$ and base $k$ in Python.
We build the automaton as described and afterwards run Dijkstra's algorithm using the symbols of $B$ as edge weights to get the shortest $B$ from the start to accept state.
Computing the automaton is $O(k^2N^3)$ as it has $O(kN^3)$ states with $O(k)$ transitions out of each state.
Given the size of the automaton and a binary heap handling the Dijkstra's algorithm, our algorithm runs in $O(k^2N^3 \log(k^2N^3))$.
The existence of $A$ and $B$ can be shown in $O(k^2N^3)$ with a simple breadth first search of the automaton for the accept state but due to the nondeterminism and variability in the difference of lengths, it can't guarantee a minimal example. 

The code is available at

\url{https://github.com/josephmeleshko/Palindrome-Ratio-Set-Automata-Generator}

This guaranteed decision algorithm can prove that there are no solutions for a variety of integers that the heuristic algorithm fails to determine. For example, our algorithm was able to prove that there are no solutions to the equation $N = A/B$ for $N \in \{ 2551, 14765, 15247, 17093, 19277, 19831 \}$.

Let $\qpal = \{ 1,3,5,7,9,11,13,15,17,19,21,27,31,\ldots \}$
be the set of integers representable as the
quotient of palindromic numbers.
With this code  we computed the data in Table~\ref{tab3} showing
the distribution of elements of $\qpal$
according to the number of bits.
\begin{table}[H]
    \centering
    \begin{tabular}{c|c}
        $i$ & $| \qpal \, \cap \, [2^{i-1},2^i)|$  \\
         \hline
         1 & 1 \\
         2 & 1 \\
         3 & 2 \\
         4 & 4 \\
         5 & 5 \\
         6 & 10 \\
         7 & 17 \\
         8 & 33 \\
         9 & 55 \\
         10 & 98 \\
         11 & 165 \\
         12 & 309 \\
         13 & 571 
    \end{tabular}
    \caption{Number of $i$-bit numbers representable as the quotient of palindromes}
    \label{tab3}
\end{table}
This numerical data suggests that perhaps
roughly $0.34 \cdot 1.76^i$ $i$-bit numbers are representable.

We can easily prove the following lower bound
on the number of representable $i$-bit
numbers.

\begin{theorem}
There are $\Omega( \sqrt{2}^i)$ $i$-bit
integers representable as the quotient of
palindromic numbers.
\end{theorem}

\begin{proof}
Every $i$-bit palindromic number $N$
can be written as $N/1$, and there
are $\Omega(2^{i/2})$ of them.
\end{proof}

Even the following seems hard to prove.
\begin{conjecture}
The set of integers representable as quotients of palindromic numbers is of zero density.
\end{conjecture}

\subsection{Size of smallest representation for palindromes}

The sequence
$$ 1,1,1,1,1,3,5,1,1,27,1,1,1,1,5,3,1,\ldots$$
of the minimal size of denominators $B$ for those $N$ having
a representation $A/B$ as a quotient of palindromes forms sequence
\seqnum{A305470} in the OEIS.

Suppose $N = A/B$ for palindromic numbers $A, B$.  We can use our algorithm
based on finite automata to upper bound the size of the numerator and denominator of the smallest such
representation using the following simple idea:
\begin{proposition}
If an NFA of $t$ states accepts the base-$k$ representation of the first
halves of strings $(A)_k$ and $(B)_k$ for palindromic numbers $A, B$
 such that $N = A/B$, then $A, B < k^{2t-1}$.   
\end{proposition}

\begin{proof}
By the pigeonhole principle applied to the sequence of states traversed
by an input, if an NFA of $t$ states accepts at least one string, then it must
accept a string of length at most $t-1$.  Hence if we have an NFA as given
in the hypothesis, it must accept at least one pair of inputs in parallel
of length $\leq t-1$.   Thus $|(A)_k|, |(B)_k| \leq 2(t-1)+1 = 2t-1$, and
so $A, B < k^{2t-1}$.
\end{proof}

A naive bound on the size of the automata $M_{N, k}$ observes that each of the three phases has unique states that are characterized by one of two maximum numbers of saved symbols $s$, up to $\ceil{\log_k N}$ saved symbols each taking one of $k$ values, and two carries each ranging from $0$ to $N-1$.
This means there are at most $$3 \cdot 2 \cdot k^{\ceil{\log_k (N)}} \cdot N^2 \leq 6 \cdot k\cdot N \cdot N^2 \in O(k\cdot N^3)$$
states in $M_{N, k}$.

More strongly, we have that the loading phase takes at most $\ceil{\log_k N}$ transitions since the automaton adds one symbol to the queue of saved symbols at a time.
The shifting phase takes at most $k^{\ceil{\log_k (N)}} \cdot N^2$ transitions since the automaton at worst goes through every state once.
The unloading phase takes at most $\ceil{\frac{\ceil{\log_k N}}{2}} = \ceil{\frac{\log_k N}{2}}$ transitions since the automaton removes two symbols from the queue of saved symbols at a time.
However, the unloading phase can also require an extra check that implicitly uses the central symbol of $(A)_k$ if $\ceil{\log_k N}$ is even and both $(A)_k$ and $(B)_k$ have an extra central symbol.

Since each transition adds two digits to $A$ and the unloading phase can may implicitly use one additional symbol from $(A)_k$, we have therefore shown:
\begin{theorem}
If there exists an $A$ and $B$ such that $(A)_k$ and $(B)_k$ are palindromes and $A/B = N$, then for the smallest $A$, $$\abs{(A)_k} \leq 2 \cdot \left( \ceil{\log_k N} + k^{\ceil{\log_k N}} \cdot N^2 + \ceil{\frac{\log_k N}{2}}\right) + 1.$$
\end{theorem}

Record-setting values of the smallest
representation $A, B$ are given in Table~\ref{largest}.
\begin{table}[H]
\centering
\resizebox{\columnwidth}{!}{%
\begin{tabular}{c|c|c}
$N$ & $A$ & $B$ \\
\hline
1 & 1 & 1 \\
11 & 33 & 3 \\
13 & 65 & 5 \\
19 & 513 & 27 \\
53&           3339&             63          \\
71&          54315&            765          \\
79&         888987&          11253        \\
149&        5887437&          39513       \\
319&      224725611&         704469     \\
575&   147606740625&      256707375   \\
1823&394070635302093&   216166009491\\
2597 & 96342506397593044197 & 37097615093412801 \\
5155 & 324903223321029232798074465 & 63026813447338357477803 \\
10627 & 9300753824529071312360470246068903 & 875200322247960036921094405389\\
22331 & 79377444895975693055708664734623129867563975 & 3554585325152285748766677029001080554725
\end{tabular}
}
\caption{Record-setting values of smallest representation $N = A/B$ in palindromic numbers.}
\label{largest}
\end{table}

\begin{conjecture}
The size of the smallest solution to $N = A/B$ in palindromes $A, B$, if it exists,
is not bounded by any polynomial in $N$.
\end{conjecture}

The available numerical data suggest that perhaps the smallest
solution, when one exists, is bounded by $N^{O(\log N)}$.

\subsection{Infinitely many integers with no representation}

Since $2^n +1$ is a palindrome for every $n \geq 1$, it is clear
that infinitely many integers belong to $\qpal$.  We now prove
that there are infinitely many odd integers in
the complement $\Enn - \qpal$.

\begin{theorem}
There are infinitely many odd positive integers $N$ for which there is no
solution to $N = A/B$ in palindromes $A, B$.
\end{theorem}

\begin{proof}
We prove that if $5 \cdot 2^k < N <
6 \cdot 2^k$ and $N \equiv \modd{1} {8}$, then
$N$ has no representation.

We prove this by considering the four possibilities for the first three bits of $A$:
\begin{itemize}
\item $A = 100\dots 001$. Then $A\equiv 1\pmod 8$ and $4\cdot 2^{j+3} < A < 5\cdot 2^{j+3}$ for some $j$. Hence $\frac 23\cdot 2^{j-k} < B < 2^{j-k}$, i.e., $B$ starts with $101$, $110$ or $111$. Therefore $B$ ends with $101$, $011$ or $111$, i.e., $B\equiv 5, 3$ or $7\pmod 8$, i.e., $A\equiv 5, 3$ or $7\pmod 8$, a contradiction. 

\item $A = 101\dots 101$. Then $A\equiv 5\pmod 8$ and $5\cdot 2^{j+3} < A < 6\cdot 2^{j+3}$ for some $j$. Hence $\frac 56\cdot 2^{j-k} < B < \frac 65\cdot 2^{j-k}$, i.e., $B$ starts with $110$, $111$ or $100$. Therefore $B$ ends with $011$, $111$ or $001$, i.e., $B\equiv 3, 7$ or $1\pmod 8$, i.e., $A\equiv 3, 7$ or $1\pmod 8$, a contradiction. 

\item $A = 110\dots 011$. Then $A\equiv 3\pmod 8$ and $6\cdot 2^{j+3} < A < 7\cdot 2^{j+3}$ for some $j$. Hence $2^{j-k} < B < \frac 75\cdot 2^{j-k}$, i.e., $B$ starts with $100$ or $101$. Therefore $B$ ends with $001$ or $101$, i.e., $B\equiv 1$ or $5\pmod 8$, i.e., $A\equiv 1$ or $5\pmod 8$, a contradiction. 

\item $A = 111\dots 111$. Then $A\equiv 7\pmod 8$ and $7\cdot 2^{j+3} < A < 8\cdot 2^{j+3}$ for some $j$. Hence $\frac 76\cdot 2^{j-k} < B < \frac 85\cdot 2^{j-k}$, i.e., $B$ starts with $100$, $101$ or $110$. Therefore $B$ ends with $001$, $101$ or $011$, i.e., $B\equiv 1, 5$ or $3\pmod 8$, i.e., $A\equiv 1, 5$ or $3\pmod 8$, a contradiction. 
\end{itemize}
\label{samin}
\end{proof}

In fact, we've proved something more:

\begin{corollary}
The set of unrepresentable $N$ has positive density in the natural numbers.
\end{corollary}

\begin{proof}
From the result above,
a number $N$ is unrepresentable if $N \equiv \modd{1} {8}$ and the first
three bits of $N$ in base $2$ are $101$.   Let us count the number $f(x)$
of integers $\leq x$ satisfying these two conditions.  Clearly
$f(x)/x$ achieves a local minimum when $x = 5 \cdot 2^n$.
In this case $f(x) = 2^{n-3}-1$.   It follows that
$\liminf_{x \rightarrow \infty} f(x)/x = 1/40$.
\end{proof}

\begin{remark}
This bound $1/40$ for the lower density can easily be improved by considering other intervals.
\end{remark}

\subsection{Infinitely many different representations}

\begin{theorem}
Suppose there is one solution in palindromes $A,B$
to the equation $N = A/B$.  Then there are infinitely many
solutions.
\end{theorem}

\begin{proof}
Suppose there is one solution $N = A/B$.
Let $d = | (A)_2 | - | (B)_2 |$.  
For each $i \geq 0$ define $A_i = [ (A)_2 0^i (A)_2 ]_2$
and $B_i = [ (B)_2 0^{i+d} (B)_2 ]_2$.
Then $A_i$ and $B_i$ are clearly palindromic numbers,
and $N = A_i/B_i$.
\end{proof}

\subsection{Rational solutions to $p/q = A/B$ in palindromes}

Our automaton method, discussed in Section~\ref{decision2}, can be modified to get a solution $A, B$ in palindromes where $A/B = p/q$ for integers $p/q$.
Instead of computing $N \cdot B = A$, the automaton computes $p \cdot B = q \cdot A$.
For simplicity, we assume that $p > q$ as $p = q$ is trivial and if $p < q$ then solutions to $p/q$ can be derived from the solutions to $q/p$.

The structure of the automaton is similar but each state has a few modifications.
In place of the right carry, we have a carry $c_{A,\ell}$ for $A$ and a carry $c_{B, \ell}$ for $B$.
At each step, the automaton verifies
$$
q \cdot A_\ell + c_{A, \ell-1} = p \cdot B_\ell + c_{B, \ell-1} \mod k.
$$
Let $m$ be the remainder of $p \cdot B_\ell + c_{B, \ell-1}$ divided by $k$.
The automaton then computes
$$
c_{A, \ell} = \frac{q \cdot A_\ell + c_{A, \ell-1} - m}{k}
$$
and
$$
c_{B, \ell} = \frac{p \cdot B_\ell + c_{B, \ell-1} - m}{k}.
$$
We get familiar bounds on the size of the carries, $0 \leq c_{A, \ell} < q$ and $0 \leq c_{B, \ell} < p$.

There is still just a single left carry, but it has to implicitly track the left carry for both $A$ and $B$.
To accomplish this, it tracks the difference between the left carry of $B$ and the left carry of $A$.
Let $c_\ell = c_{B,\ell} - c_{A, \ell}$.
We can then derive an equation for computing $c_{\ell-1}$ from $c_\ell$.
\begin{align*}
c_\ell &= c_{B, \ell} - c_{A, \ell}\\
c_\ell &= \frac{p \cdot B_\ell + c_{B, \ell-1} - m}{k} - \frac{q \cdot A_\ell + c_{A, \ell-1} - m}{k}\\
k \cdot c_\ell &= p \cdot B_\ell + c_{B, \ell-1} - m - q \cdot A_\ell - c_{A, \ell-1} + m\\
k \cdot c_\ell &= p \cdot B_\ell + c_{B, \ell-1} - q \cdot A_\ell - c_{A, \ell-1}\\
k \cdot c_\ell - p \cdot B_\ell + q \cdot A_\ell &= c_{B, \ell-1} - c_{A, \ell-1}\\
k \cdot c_\ell - p \cdot B_\ell + q \cdot A_\ell &= c_{\ell-1}
\end{align*}
From the bounds on $c_{A, \ell}$ and $c_{B, \ell}$ we get that $-q < c_\ell < p$.

The remaining structure is essentially identical.
An accepting state is one where the left carry is equal to the difference of the two right carries.
(With the nondeterminism around the middle symbols handled as usual.)
The automaton nondeterministically chooses the difference in size of $A$ and $B$ to be either the floor or ceiling of $\log_k \frac{p}{q}$.
Since $1 < p/q < k$ can be a valid input, $A$ and $B$ could have the same length.
However, this only simplifies the construction as we ignore the loading and unloading phase entirely since all of the symbols in the equations line up perfectly.

Given the constraints on all the information we track, (and $p > q$,) there are at most
$$6 \cdot (p+q-1) \cdot p \cdot q \cdot k^{\ceil{\log_k \frac{p}{q}}} \in O(kp^3)$$
states in the new automaton which gives analogous bounds for computation of the minimal $p/q = A/B$.

\begin{conjecture}
For all odd numbers $p > 1$, $p \not=23$, there exists
an odd number $q <p$ such that $p/q = A/B$ has a solution
in palindromes $A, B$.
\end{conjecture}

We have verified this conjecture for $p < 1000$.   For $p = 23$ we can definitively
prove, using our automaton method, that there is no odd $q < 23$ such that
$p/q = A/B$ has a solution in palindromes.

Sometimes the smallest solution to $p/q=A/B$ can be quite large.
For example, the smallest solution to $A/B = 979/765$ in
palindromic numbers is

$$ {{\scriptstyle 435964577851526887677597179561025269848009167916543881959761365529045212378773108135544954987} \over {\scriptstyle 340666907105636434191380840004274087266319727738668099794910566526782009672892163149838499045}} \ .$$

\section{Antipalindromic numbers}

In this section we treat the same six problems for the antipalindromic numbers.   

\subsection{Denseness}

\begin{theorem}
The set ${\tt APAL}/{\tt APAL}$ is dense in the positive reals.
\label{thm14}
\end{theorem}

\begin{proof}
The proof is analogous to the proof of Theorem~\ref{one}.  We just outline
the basic idea.
Let $\alpha$, $\beta$, and $k$ be as in that proof.

There are two cases:  $k$ odd and $k$ even.

If $k$ is odd, for a given $n$ define $\gamma = \lfloor 2^n \beta \rfloor$.  Set
$A = [(\gamma)_2 \overline{ (\gamma_2)^R} ]_2$ and
$B = [1 0^c 1^c 0]_2$ for $c = n+(k-1)/2$.   Then $A$ and $B$
are both antipalindromic numbers, and $A/B$ is an arbitrarily
good approximation to
$\alpha$, as $n \rightarrow \infty$.

If $k$ is even, define $\gamma = \lfloor 2^n/\beta \rfloor$.  Set 
$B = [(\gamma)_2 \overline{ (\gamma_2)^R} ]_2$ and
$A = [1 0^{n-k/2} 1^{n-k/2} 0]_2$.   Then
$A/B$ is an arbitrarily
good approximation to
$\alpha$, as $n \rightarrow \infty$.
\end{proof}

\begin{remark}
Let $a_1< a_2 < a_3 < \cdots$ be the antipalindromic numbers.
Here the criterion of Theorem~\ref{nark} would not suffice to prove
Theorem~\ref{thm14}, since $\limsup_{n \rightarrow \infty} a_{n+1}/a_n = 2$.
\end{remark}

\subsection{Quotients of antipalindromic numbers}

The set $\qapal = \{ 1,5,6,15,17,18,19,20,21,24,26,\ldots \}$ of
integers representable as the quotient of two antipalindromic
numbers, forms sequence \seqnum{A351172} in the OEIS.  The
set $\Enn - \qapal = \{ 2,3,4,7,8,9,10,11,12,13,14,16,22, \ldots \}$
of unrepresentable integers forms sequence 
\seqnum{A351173}.

\subsubsection{Decision algorithm}
\label{antipal-decision}

We can verify if a given $N$ is representable as the quotient of antipalindromic numbers $A$ and $B$ using an analogous method to the algorithm given in Section~\ref{decision2}.
We build a similar automaton to the automaton in Section~\ref{decision2}, though it interprets the input $\langle a, b\rangle$ as the quotient of $A = a \sigma_a \overline{a}^R$ and $B = b \sigma_b \overline{b}^R$.
This interpretation is dependent on the base $k$, as the middle character must be such that $\sigma = \overline{\sigma}$. 
If $k$ is odd then, $\sigma_x \in \{\epsilon, (k-1)/2\}$ and if $k$ is even then $\sigma_x = \epsilon$.
When the automaton interprets the input as above and computes accordingly, it accepts antipalindromes that have quotient $N$.
The algorithm also achieves the same asymptotic bounds.
Thus we have
\begin{theorem}
There is an algorithm that, given a natural number $N$, can determine
if there exist antipalindromes in base $k$ $A, B$ such that $N = A/B$ in
$O(k^2N^3)$ time.
\end{theorem}

With our algorithm we computed the number of representable $i$-bit integers.
\begin{table}[H]
    \centering
    \begin{tabular}{c|c}
        $i$ & $| \qapal \, \cap \, [2^{i-1},2^i)|$  \\
         \hline
         1 & 1 \\
         2 & 0 \\
         3 & 2 \\
         4 & 1 \\
         5 & 8 \\
         6 & 4 \\
         7 & 24 \\
         8 & 17 \\
         9 & 75 \\
         10 & 50 \\
         11 & 247 \\
         12 & 165 \\
         13 & 903
    \end{tabular}
    \caption{Number of $i$-bit numbers representable as the quotient of antipalindromes}
    \label{tab4}
\end{table}
The available data suggest that perhaps there are roughly
$.12 \cdot 1.81^i$ $i$-bit solutions for $i$
even, and $.36 \cdot 1.81^i$ for $i$ odd.

We can prove the following lower bound.
\begin{theorem}
There are $\Omega(\sqrt{2}^i)$ $i$-bit integers
representable as the quotient of antipalindromes.
\end{theorem}

\begin{proof}
If $i$ is odd, then we can get
$O(\sqrt{2}^i)$ representable integers of $i$ bits
by taking $A$ to be an antipalindromic
number of $i+1$ bits, and $B = 2$.

If $i$ is even, say $i = 2j$, we have to work a bit
harder.  First we observe that if $x,y$
are arbitrary binary strings of $j$ bits
and $y$ ends in $1$, and
$A = [x \, y \, \overline{x} \, \overline{y}]_2$, then $A$
is divisible by $2(2^{2j}-1)$.

To see this, note that
\begin{align*}
A &= [x]_2 \cdot 2^{3j} + [y]_2 \cdot 2^{2j} + [\overline{x}]_2 \cdot 2^j + 
[\overline{y}]_2 \\
&= [x]_2 \cdot 2^{3j} + [y]_2 \cdot 2^{2j} + (2^j - 1 - [x]_2)\cdot 2^j + 
(2^j - 1 - [y]_2) \\
&= (2^{2j} - 1)([x]_2 \cdot 2^j + [y]_2 + 1),
\end{align*}
and observe that the second factor of the
last line is even if $[y]_2$ is odd.

Now take $y = x^R$, so that $x$ starts with $1$.  Then $A$ is an antipalindromic
number (because its base-$2$ representation is
given by $x\, x^R \, \overline{x} \, \overline{x}^R $).   From the previous paragraph we see that $A$ is divisible
by $2(2^{2j} - 1)$.  Since $2^{2j} - 1$ is
divisible by $3$, it follows that
$A$ is divisible by the antipalindromic number $B = 2(2^{2j} - 1)/3 =
[(10)^j]_2$.

Finally we need to estimate the number
of these quotients $A/B$ that have $i$ bits.
Suppose $2^{j-1} \leq [x]_2  \leq (2^{j+1} - 5)/3$.   Then
\begin{align*}
    2^{4j-1} & \leq A \leq 2^{3j} (2^{j+1} - 2)/3 \\
    {{3 \cdot 2^{4j-1}} \over 
    {2 \cdot (2^{2j} - 1)}} & \leq A/B \leq 
    {{(2^j -1)2^{3j}} \over {2^{2j}-1}} = {{2^{3j}} \over {2^j+1}} < 2^{2j},
\end{align*}
so $A/B$ has $2j$ bits.  Thus there are
at least 
$$(2^{j+1} - 5)/3 - 2^{j-1} + 1 = 2^{j+1}/12 -2/3$$ numbers of $2j$
bits that are representable as quotients
of antipalindromic numbers.
\end{proof}

Even the following seems hard to prove.
\begin{conjecture}
The set of integers representable as quotients of antipalindromic numbers is of zero density.
\end{conjecture}

\subsection{Size of the smallest representation}

With our algorithm, we were able to compute the 
record-setting values of $A,B$ given in Table~\ref{tab6}.
\begin{table}[H]
\centering
\resizebox{\columnwidth}{!}{%
\begin{tabular}{c|c|c}
$N$ & $A$ & $B$ \\
\hline
5 & 10 & 2 \\[.07in]
15 & 150 & 10 \\[.07in]
18 & 936 & 52 \\[.07in]
59 & 52140188 & 883732 \\[.07in]
66 & 65099232 & 986352 \\[.07in]
83 & 206712630902722 & 2490513625334 \\[.07in]
343 & 841469573210301602 & 2453264061837614 \\[.07in]
835 & 180616526119856633856230 & 216307216910007944738 \\[.07in]
991 & 200428779760870700728006297372550 & 202249020949415439685172853050 \\[.07in]
1268 & 75547761517760569279087608058268904 & 59580253562902657160163728752578 \\[.07in]
1290 & 4395923940796125166581803114404301293837667532540 & 3407692977361337338435506290235892475843153126 \\[.07in]
1952 & 1586681992762659022973996447792006955471260017904473853156544 & 812849381538247450294055557270495366532407796057619801822 \\[.07in]
4091 & 102232724919890518755288528068181989159740544137704480818962816 & 24989666321166100893495118080709359364395146452628814670976 \\[.07in]
4460 & 388987104335534771520764071813224655554298718228899978912430000 & 87216839537115419623489702200274586447152178975089681370500  \\[.07in]
4640 
& 85112365674283227507265261996365447811182320230460498
& 18343182257388626617945099568182208579996189704840624 \\
& 83220363630941564530051147747252794193043200 & 74831974920461544079752402531735515989880 \\[.07in]
4848 
& 16307148112492799707206815760673202828585190069605262924 & 33636856667683167712885346040992580091966151133674222204  \\
& 53647949068964670350753495389303768493316355031391840704 & 07689663921131745773006384878926915208985880840329704424 \\
& 83231286976 & 1590612 \\[.07in]
5840  
& 43493875233140378950672024766781801439773086758844870 & 74475813755377361216904151997914043561255285545967244 \\
& 87734362685028948495519190020221259652712554180379503 & 65298566241487925506026010308598047350535195514348464 \\
& 2037061443716557960233120 & 389907781458314719218 \\[.07in]
6624 
& 33301854653004018709445764603598238453897624842252171 & 50274539029293506505805804051325843076536269387457987 \\
& 14065184395426890419279201949902950413217648251363098 & 83310966780535764521858698595867980696282681538893566 \\
& 49496826284424060000589698848419903839832345753230313 & 56849073497016998793160777247010724395882164482533686 \\
& 683200744371137665623242712430446372631626633794372416 & 11594315273420541307856689678509416158156194715334 
\end{tabular}}
\caption{Record-setters for smallest solutions $A, B$ to $N = A/B$ in antipalindromes.}
\label{tab6}
\end{table}

The size of smallest solutions is somewhat larger than in the case for palindromes, which is not too surprising, since there are fewer antipalindromes in base $2$ than palindromes (since the length of an antipalindrome must be even).

\begin{conjecture}
The size of a smallest solution to $N = A/B$, if it exists, is not bounded by polynomial in $N$.
\end{conjecture}

\subsection{Numbers representable as quotients of antipalindromes}

\begin{theorem}
There is an infinite set of integers representable as a quotient of antipalindromes in base 2.
\end{theorem}

\begin{proof}
Integers of the form $N = 2^{2n+1} - 2^n = (1^{n+1}0^{n})_2$ are representable as a quotient of antipalindromes.
We have that $A/B = N$ for $A = 2^{2n+2} - 2^{n+1} = (1^{n+1}0^{n+1})$ and $B = 2 = (10)_2$ which are both antipalindromes.
\end{proof}

\begin{theorem}
There is an infinite set of integers that are not representable as a quotient of antipalindromes in base 2.
\end{theorem}

\begin{proof}
Integers of the form $N = 2^{2n+1}$ aren't representable as a quotient of antipalindromes.
There are no antipalindromes of odd length in base 2, as the middle digit $\sigma$ must equal $2-1-\sigma = 1-\sigma$ which has no solutions in $\{0, 1\}$.
Given any antipalindromic $B$ of length $2i$, $N\cdot B$ is of length $2i+2n+1$ which is odd and not an antipalindromes.
Therefore, there are no antipalindromic $A$ and $B$ such that $A/B = N$.
\end{proof}

\begin{theorem}
There are infinitely many $N$ for which there is no
representation $N = A/B$ with $A, B \in \apal$.
\end{theorem}

\begin{proof}
We show that if 
\begin{equation}
40 \cdot 4^n < N < 48 \cdot 4^n
\label{ineq}
\end{equation}
for $n \geq 0$,
and $N \equiv \modd{1} {4}$, then there is
no representation $N = A/B$.

Suppose such a representation exists. 
Notice that $A$ being an antipalindrome
means that $(A)_2$ has an even number of
bits; that is, that $2 \cdot 4^i \leq A
< 4^{i+1}$ for some $i$.

Further, the inequality $40 \cdot 4^n < N < 48 \cdot 4^n$ implies that the first
three bits of $(N)_2$ must be $101$.   
Since $(B)_2$ is an antipalindrome, $B$ must be even.  If $B = 2$, then $N = A/B$
implies that $4^i \leq A < 2 \cdot 4^i$ 
for some $i$, contradicting \eqref{ineq}.

So $B$ is at least $4$, and hence 
$(B)_2$ has at least three bits.
We now claim that the first three bits of
$(A)_2$ must be the same as the first
three bits of $(B)_2$.  To see this,
suppose the first three bits of $(B)_2$
are $1bc$.  Since $(B)_2$ is an
antipalindrome, the last three bits of
$(B)_2$ must be $\overline{c} \overline{b} 0$.
Now $N \equiv \modd{1} {4}$, so by
considering $BN \bmod 8$, we see that the
last $3$ bits of $A = BN$ must also be
$\overline{c} \overline{b} 0$.  Since
$(A)_2$ is an antipalindrome, the first
three bits of $(A)_2$ must also be $1bc$,
as claimed.

There are now four possibilities to check.
These are summarized in Table~\ref{tab5} 
below, where $j$ is some positive integer.
\begin{table}[H]
\begin{center}
\begin{tabular}{cc|c}
$b$ & $c$ & inequality \\
\hline 
0 & 0 & $(4/5) 4^j < A/B < (5/4) 4^j$ \\
0 & 1 & $(5/6) 4^j < A/B < (6/5) 4^j$ \\
1 & 0 & $(6/7) 4^j < A/B < (7/6) 4^j$ \\
1 & 1 & $(7/8) 4^j < A/B < (8/7) 4^j$
 \end{tabular}
    \end{center}
    \caption{Inequalities.}
    \label{tab5}
\end{table}
In each case these contradict
\eqref{ineq}.
So $N$ is not representable.
\end{proof}

\begin{corollary}
The lower density of unrepresentable numbers is $\geq 1/60$.
\end{corollary}

\begin{proof}
By the previous proof, a number $N$ is unrepresentable if it has an even number
of bits, is congruent to $1$ (mod $4$), and begins with $101$.  If we let
$g(x)$ be the number of such numbers $\leq x$, then $g(x)/x$ clearly has
a local minimum at $x = 40 \cdot 4^n$, and for such $x$ we have
$g(x) = (2 \cdot 4^n - 2)/3$.  The bound of $1/60$ now follows.
\end{proof}

\subsection{Number of solutions}

Another advantage of the finite automaton method is that for a given $N$ we can determine if there are infinitely many solutions to $A/B = N$ in antipalindromes, or whether there are any fixed number of solutions.

Given the finite automaton constructed in Section~\ref{antipal-decision} for antipalindromes, we first remove all states from which we cannot reach a final state.
(The construction ensures that all states are reachable from the start state.)  The resulting automaton has a cycle if and only if there are infinitely many solutions.

We used this idea to compute the first few terms of the relevant sets.   This gives us sequence \seqnum{A351175}, those $N$
for which there are infinitely many solutions:
$$ 1, 6, 15, 18, 19, 20, 24, 28, 51, 59, 61, 63, 66, 67, 68, 71, 72, 74, \ldots,$$
sequence \seqnum{A351176}, those $N$ for which there is at least one solution, but only finitely many:
$$ 5, 17, 21, 26, 65, 69, 70, 85, 89, 92, 102, 106, 116, 219, 221, 233, 239, 245, 249, 257, \ldots,$$
and sequence \seqnum{A351325}, those $N$ for which there is exactly one solution:
$$ 5, 21, 26, 69, 85, 89, 92, 102, 106, 116, 219, 221, 233, 239, 245, \ldots .$$

\begin{theorem}
There are infinitely many integers $N$ for which there are
infinitely many solutions to the equation $N = A/B$ for 
antipalindromes $A, B$.
\end{theorem}
\begin{proof}

Let $N = 2^{2n+1} - 2^n$ for $n \geq 1$.
Define
$$B_i = [1 (0^{n+2} 1^{n+2})^i 0]_2 = 2^{2ni + 4i + 1} + (2^{n+2} - 1) \cdot \sum\limits_{j = 0}^{i-1} 2^{2nj + 4j + 1}$$
for $i \geq 0$.
Clearly $B_i$ is an antipalindrome.
We now compute $A_i = N \cdot B_i$.
\begin{align*}
N\cdot B_i &= (2^{2n+1} - 2^n) \cdot \left( 2^{2ni + 4i + 1} + (2^{n+2} - 1) \cdot \sum\limits_{j = 0}^{i-1} 2^{2nj + 4j + 1} \right)\\
&= (2^{2ni + 4i + 2n + 2} - 2^{2ni + 4i + n + 1}) + (2^{3n+3} - 2^{2n+2} - 2^{2n+1} + 2^n) \cdot \left( \sum\limits_{j = 0}^{i-1} 2^{2nj + 4j + 1} \right)\\
&= (2^{2ni + 4i + 2n + 2} - 2^{2ni + 4i + n + 1}) + (2^{2n+3} - 2^{n+2} - 2^{n+1} + 1) \cdot 2^n \cdot \left( \sum\limits_{j = 0}^{i-1} 2^{2nj + 4j + 1} \right)\\
&= (2^{2ni + 4i + 2n + 2} - 2^{2ni + 4i + n + 1}) + (2^{2n+3} - 2^{n+2} - 2^{n+1} + 1) \cdot \left( \sum\limits_{j = 0}^{i-1} 2^{2nj + 4j + n + 1} \right)\\
&= [1^{n+1}0^{2ni + 4i + n + 1}]_2 + [1^n 0 1 0^n 1]_2 \cdot [(10^{2n+3})^{i-1} 1 0^{n+1}]_2\\
&= [1^{n+1}0^{2ni + 4i + n + 1}]_2 + [(1^n 0 1 0^n 10)^{i-1} 1^n 0 1 0^n 1 0^{n+1}]_2\\
&= [1^{n+1} 0 (1^n 0 1 0^n 10)^{i-1} 1^n 0 1 0^n 1 0^{n+1}]_2\\
&= [1^{n+1} (0 1^n 0 1 0^n 1)^i 0^{n+1}]_2\\
&= A_i.
\end{align*}
Thus $A_i$ is also an antipalindrome for each $i \geq 0$.
Therefore, we have an infinite set of representations $A_i/B_i = N$ where $A_i$ and $B_i$ are antipalindromes for each $N = 2^{2n+1} - 2^n$.
\end{proof}

\begin{theorem}
There are exactly $2^{i-1}$ solutions to $N = A/B$ for $N = 4^i + 1$ and $A, B$ antipalindromes.
\end{theorem}

\begin{proof}
Let $N = 4^i + 1 = [10^{2i-1}1]_2$.
Consider an antipalindrome $B$.
Let $(B)_k = \beta$ and $|\beta| = \ell$.

If $\beta$ has length $\ell < 2i$, then $(BN)_k = (A)_k = \beta0^{ 2i - \ell }\beta$.
Since antipalindromes in base 2 have even length, the center of $(A)_k$ is at least two zeros which means that $A$ isn't an antipalindrome. 

If $\beta$ has length $\ell = 2i$, then $(BN)_k = (A)_k = \beta \beta$.
Here, $A$ is an antipalindrome since $\overline{(\beta \beta)^R} = \overline{\beta}^R\overline{\beta}^R = \beta\beta$.

If $\beta$ has length $\ell > 2i$, then $(BN)_k$ can be viewed as the binary addition of $[\beta0^{2i}]_2 + [\beta]_2$.
Since $\beta$ was sufficiently long, there is some non-trivial overlap in the addition.
Let $j = 2i - \ell$.
The overlap has length $\ell - j$ and there are $j$ symbols of $\beta$ on each side of the overlap.

\begin{figure}[H]
\begin{center}
\begin{tabular}{cc|c|c}
  & $\beta[1:j]$ & $\beta[j + 1:\ell]$     & $0^j$\\
+ & $0^j$        & $\beta[1 : \ell - j]$   & $\beta[\ell - j + 1 : \ell]$
\end{tabular}
\end{center}
\caption{Piecewise addition of $[\beta0^{2i}]_2 + [\beta]_2$.}
\label{AdditionTable1}
\end{figure} 

Since $B$ is an antipalindrome, we get that $\beta[1:j] = \overline{\beta[\ell - j + 1: \ell]}^R$.
This means for $[\beta0^{2i}]_2 + [\beta]_2$ to be an antipalindrome the overlap region must not overflow to the left.
We have additional information that further constrains this addition.
We know that $\beta[1] = 1$ which implies that $\beta[\ell] = \overline{\beta[1]} = 0$.
Additionally, we know that the overlap region can't overflow so $\beta[j+1] = 0$ which subsequently implies that $\beta[\ell-j] = \overline{\beta[j+1]} = 1$.
As well, the remaining addition $\beta[j + 2:\ell - 1] + \beta[2 : \ell - j - 1]$ must not overflow either.

\begin{figure}[H]
\begin{center}
\begin{tabular}{cc|c|c|c|c}
  & $\beta[1:j]$ & 0 & $\beta[j + 2:\ell - 1]$     & 0 & $0^j$\\
+ & $0^j$        & 1 & $\beta[2 : \ell - j - 1]$   & 1 & $\beta[\ell - j + 1 : \ell]$\\
\hline
= & $\beta[1:j]$ & 1 & $\beta'$                    & 1 & $\beta[\ell - j + 1 : \ell]$
\end{tabular}
\end{center}
\caption{Piecewise addition of $[\beta0^{2i}]_2 + [\beta]_2$ with constraints.}
\label{AdditionTable2}
\end{figure} 

From the result of the addition we see that we have a 1 at $j+1$ symbols from the front and a 1 at $j+1$ symbols from the back.
Therefore, this can't be an antipalindrome.

Overall, given an antipalindrome $B$, $BN$ is an antipalindrome if and only if $(B)_k$ has length $2i$.
There are $2^{i-1}$ antipalindromes of length $2i$, so for $N = 4^i + 1$ there are exactly $2^{i-1}$ solutions to $N = A/B$ for $A$ and $B$ antipalindromes.
\end{proof}

\begin{theorem}
There are infinitely many integers $N$ such that
$N = A/B$ has exactly one solution in antipalindromes
$A,B$.
\end{theorem}

\begin{proof}
Consider $N$ of the form $(2^{2n} -1)/3$ for 
$n \geq 2$.   Clearly $(2N)_2 = (10)^n$, so
$2N$ and $2$ are both antipalindromes.   This gives
one solution to $N = A/B$.

Now let us assume there is another solution
to $N = A/B$ with $A, B$ antipalindromes.  Since
$B > 2$, and the next larger antipalindrome is $10$,
we see that $B$ has at least $4$ bits.   
Choose $k \geq 1$ such that $4 \cdot 2^k \leq B < 8 \cdot 2^k$.

Note that
$5\cdot 2^{2n-3} \leq N < (16/3) \cdot 2^{2n-3}$.
We can use this inequality together with $A = BN$
to determine the first three bits of $A$.
They are summarized in Table~\ref{tab19}, 
where $\ell = k+2n-3$.
\begin{table}[H]
\begin{center}
\begin{tabular}{c|c|c|c}
first three & inequality& inequality & first three \\
bits of $B$ & for $B$ & for $A =BN$ & bits of $A=BN$ \\
\hline
100 & $4 \cdot 2^k \leq B < 5 \cdot 2^k$ &
    $20 \cdot 2^\ell \leq A < {80\over 3} \cdot 2^\ell$ & 101 or 110 \\[.05in]
101 & $5 \cdot 2^k \leq B < 6 \cdot 2^k$ &
    $25 \cdot 2^\ell \leq A < 32 \cdot 2^\ell$ & 110 or 111 \\[.05in]
110 & $6 \cdot 2^k \leq B < 7 \cdot 2^k$ &
    $30 \cdot 2^\ell \leq A < {112\over 3} \cdot 2^\ell$ & 111 or 100 \\[.05in]
111 & $7 \cdot 2^k \leq B < 8 \cdot 2^k$ & 
    $35 \cdot 2^\ell \leq A < {128\over 3} \cdot 2^\ell$ & 100 or 101 
\end{tabular}
\end{center}
\caption{Possibilities for first three bits of $A$.}
\label{tab19}
\end{table}
On the other hand,
if $B$ starts with
three bits $abc$, then since $B$ is an antipalindrome,
it must end with $\overline{c} \overline{b} \overline{a}$.   Since 
$N \equiv \modd{5} {8}$, one can easily check that
$A = BN$ also ends with 
$\overline{c} \overline{b} \overline{a}$.   Since 
$A$ is an antipalindrome, it must begin with
$abc$.  So the first three bits of $A$ and $B$
are the same.   This contradicts the results of
Table~\ref{tab19}, and proves there are no other
solutions.
\end{proof}

\subsection{Rational solutions to $p/q = A/B$ in antipalindromes}

Once again our automaton method for antipalindromes can be generalized to give the following result.
\begin{theorem}
There is an algorithm that, given integers $p, q \geq 1$, will decide if there is a solution to $p/q = A/B$ in antipalindromes A,B.
\end{theorem}

We used our algorithm to study the rational solutions to
$p/q = A/B$ in antipalindromes for $p>q$ and $p \leq 1000$.  Based
on our calculations, we make the following conjecture.
\begin{conjecture}
For all $p \geq 4$ there exists $q < p$ such that
$p/q = A/B$ has a solution in antipalindromes.
\end{conjecture}
We note that there are no solutions for $(p,q) \in \{ (2,1), (3,1), (3,2) \}$.

Some solutions to $p/q=A/B$ can be enormously large.  For example, the smallest solution
for $p/q = 960/527$ is $A = $
\begin{align*}
&
  1234883355213990975204467140683475994799335003626682427756930130658317 \\
& 0577845541101597875372665385744362733254798839009872167396310323997903 \\
& 5640547077917392804795250182028753800174169116477800361082899344465944 \\
& 5560841114705454770902394470289417027557405950223685182751710075724367 \\
& 6048238590480983878073501486368624181821560779594741108091349800844282 \\
& 5679592833678865846036391335428845975712764583827139150178213891564696 \\
& 4718825426930262288729775928481863474655184300859716583115484263497126 \\
& 2961706100246193708891656878945533178186000927736300244493837237642640 \\
& 9349549969820438753161560915890436797199051312068851515357512387981254 \\
& 4604809069807177738058155380014435541831993909182136704602824634226568 \\
& 3451444571619483682225669077170879824401082095216563292486986361198314 \\
& 2620371328966957512364597567158981492432747694245025717455343991855418 \\
& 3265938974040814493629275353847375559776274838299843008368743842579023 \\
& 9993356699741468657156369097163207591351729526813712761138142291367822 \\
& 0794954727600533534516312312331038829749723349859042215544591191317981 \\
& 8600650852792742320291709382397741664309047654075764338087057307850282 \\
& 7509649077719055308633225064218430763198619435136533732460140152765958 \\
& 4251780426995592541451396343086191791838699791485099128013340230974422 \\
& 4295888043536875650860208149665479650685364073997568860181548161096442 \\
& 1040420056468998183952438585617409445628800                            \\
\end{align*}
and $B = $
\begin{align*}
&
  6778995085393471290966189407710331763117182780325642077373981029759719 \\
& 6817964585005646670014527690492491254429989459981277418935995216113491 \\
& 4401753229817354251323925478428679715539449212331258232194666193057841 \\
& 4693367369268486086099602977526278890862009747582105117814075103195226 \\
& 3306476428994567747340992534544426498124609696316964207959805677551426 \\
& 1803598159882940633970606601781269054173197246634399293165820008902031 \\
& 6737718749919252355839499107395229699409188818261152492727710488156099 \\
& 5633532446143167547769824741711416509416900926219064883835960669142414 \\
& 2991800355160116905376485444523543667957292098544632797848010713188761 \\
& 4653483122795652791215082138204245109848549897281104617975922731639599 \\
& 1446992596286123963884662538219309036035106918532592241048352211994912 \\
& 6676413441308193843918155394716492151167271196532589094780898788622973 \\
& 5220310826244887897319042827891322083355175414416846514690916719157767 \\
& 1630197716289103982514651189635525006691265214904444011664593620321273 \\
& 2905636890057095548855172797900598575813585472663700495749995394006004 \\
& 5859822910643491695768029630454269344696542851020081314290408346219781 \\
& 3516511082895230704684475092115760543809087940801596635484311046954792 \\
& 6048836302361221555675894508400240357281195730340075421489898976286673 \\
& 1290968738999306958368017654934455999074863197882487388704957092685676 \\
& 966980593499127128065557431896223726923310   \ .                           
\end{align*}






\section{Going further}

We have not examined what surprises might await
us in other bases.  To give just a taste, the
smallest representation of $436$ as the
quotient of base-$10$ palindromes is
$${4062320931846767973606063797676481390232604\over
 9317249843685247645885467425863489427139} \ .$$

\end{document}